\documentclass[12pt]{amsart}
\usepackage{a4wide,enumerate,color}
\usepackage{mathrsfs} 
\allowdisplaybreaks

\let\pa\partial  
\let\na\nabla  
  
\newcommand{\N}{{\mathbb N}}  
\newcommand{\R}{{\mathbb R}} 
\newcommand{\diver}{\operatorname{div}} 
\newcommand{\en}{{\mathcal F}}
\newcommand{\E}{{\mathcal E}}
\newcommand{\diag}{\operatorname{diag}}
\newcommand{\W}{{\mathcal W}}

\newtheorem{theorem}{Theorem}   
\newtheorem{lemma}[theorem]{Lemma}   
\newtheorem{proposition}[theorem]{Proposition}   
\newtheorem{remark}[theorem]{Remark}   
\newtheorem{corollary}[theorem]{Corollary}  
\newtheorem{definition}{Definition}  
\newtheorem{example}{Example} 

\begin{document}  

\title[Displacement convexity]{Displacement convexity for the
entropy in semidiscrete nonlinear Fokker-Planck equations}

\author{Jos\'e A. Carrillo}
\address{J.A.C.: Department of Mathematics, Imperial College London, London SW7 2AZ,
United Kingdom}
\email{carrillo@imperial.ac.uk}

\author{Ansgar J\"ungel}
\address{A.J.: Institute for Analysis and Scientific Computing, Vienna University of  
	Technology, Wiedner Hauptstra\ss e 8-10, 1040 Wien, Austria}
\email{juengel@tuwien.ac.at} 

\author{Matheus C. Santos}
\address{M.C.S.: Departamento de Matem\'atica -- IMECC, Universidade Estadual de
Campinas, 13083-859, Campinas-SP, Brazil}
\email{matheus.santos@ufrgs.com}

\date{\today}

\thanks{The first author was partially supported by the Royal Society via a Wolfson 
Research Merit Award. The second author acknowledges partial support from   
the Austrian Science Fund (FWF), grants P22108, P24304, and W1245. The last author 
acknowledges the support from the S\~{a}o Paulo Research Foundation (FAPESP), 
grant $\#$2015/20962-7. We warmly thank the Institute Mittag-Leffler for providing 
a marvellous atmosphere for research while this paper has been finalized.} 

\begin{abstract}
The displacement $\lambda$-convexity of a nonstandard entropy with 
respect to a nonlocal transportation metric in finite state spaces
is shown using a gradient flow approach. 
The constant $\lambda$ is computed explicitly in terms of a priori
estimates of the solution to a finite-difference approximation of a 
nonlinear Fokker-Planck equation.
The key idea is to employ a new mean function,
which defines the Onsager operator in the gradient flow formulation.
\end{abstract}

 

\maketitle


\section{Introduction}

Displacement convexity, which was introduced by McCann \cite{Mcc97},
describes the geodesic convexity of functionals on the space of probability measures
endowed with a transportation metric. Geodesic convexity has important
consequences for the existence and uniqueness of gradient flows in the space
of probability measures \cite{AGS05,CLSS10,Ott01}. It may also provide quantitative
contraction estimates between solutions of the gradient flows \cite{CMV06} and
exponential decay estimates \cite{AGS05}. Displacement 
$\lambda$-convexity of the entropy is equivalent to a lower bound on the
Ricci curvature $\mbox{Ric}_M$ of the Riemannian manifold $M$, 
i.e.\ $\mbox{Ric}_M\ge\lambda$ \cite{LoVi09,ReSt05}. 
Furthermore, it leads to inequalities in
convex geometry and probability theory, such as the Brunn-Minkowski, Talagrand,
and log-Sobolev inequalities \cite{Vil09}. 

We are interested in the question to what extent the concept of displacement 
convexity can be extended to discrete settings, like numerical discretization
schemes of gradient flows. As one step in this direction, we show in this paper
that a certain entropy functional, related to the finite-difference approximation 
of nonlinear Fokker-Planck equations,
is displacement convex. Before making this statement more precise, let us
review the state of the art of the literature.

The study of discrete gradient flows and related topics is very recent. 
First results were concerned with Ricci curvature bounds in discrete settings
\cite{BoSt09}. 
Markov processes and Fokker-Planck equations on finite graphs were
investigated by Chow et al.\ in \cite{CHLZ12}. Maas \cite{Maa11} and Mielke
\cite{Mie13} introduced nonlocal 
transportation distances on probability spaces such that
continuous-time Markov chains can be formulated as gradient flows of the entropy,
and they explored geodesic convexity properties of the functionals. The concept
of displacement convexity was used by Gozlan et al.\ \cite{GRST14} to derive
HWI and log-Sobolev inequalities on (complete) graphs. Talagrand's inequality 
was studied in discrete spaces by Sammer and Tetali \cite{SaTe09}. 

Only few results can be found in the literature on convexity properties of
functionals for discretizations of partial differential equations.
Exponential decay rates for time-continuous Markov chains were
derived by Caputo et al.\ \cite{CDP09}.
This result implies the displacement convexity of
the Shannon entropy for discretizations of one-dimensional
linear Fokker-Planck equations, as first investigated by Mielke \cite{Mie13}
(also see the presentation in \cite[Section~5.2]{Jue16}). 
While the proof of Caputo et al.\ \cite{CDP09} is based on the Bochner-Bakry-Emery
method, Mielke \cite{Mie13} employed a gradient flow approach together with
matrix estimates. 
The nonlocal transportation metric, 
needed for the definition of displacement convexity,
is induced by the logarithmic mean,
$$
  \Lambda(s,t) = \frac{s-t}{\log s-\log t}\quad\mbox{for }s\neq t, 
	\quad \Lambda(s,s)=s,
$$
which has some remarkable properties (proved in \cite{Mie13}
and summarized in Lemma \ref{lem.Lambda} below). 
The approach of \cite{CDP09} (and \cite{FaMa16}) was extended to general
convex entropy densities $f(s)$ in \cite{JuYu15} using the mean function
\begin{equation}\label{1.Laf}
  \Lambda^f(s,t) = \frac{s-t}{f'(s)-f'(t)}\quad\mbox{for }s\neq t, \quad
	\Lambda^f(s,s) = \frac{1}{f''(s)},
\end{equation}
which becomes the logarithmic mean for $f(s)=s(\log s-1)$.

Concerning nonlinear equations, we are only aware of two results.
Erbar and Maas \cite{ErMa14} showed that a discrete one-dimensional
porous-medium equation is a gradient flow of the R\'enyi entropy function
$f(s)=s^\alpha$
with respect to a suitable nonlocal transportation metric induced by the
mean function
$$
  \Lambda^\alpha(s,t) = \frac{\alpha-1}{\alpha}\,
	\frac{s^\alpha-t^\alpha}{s^{\alpha-1}-t^{\alpha-1}}\quad\mbox{for }s\neq t,
	\quad \Lambda^\alpha(s,s) = s.
$$
However, the R\'enyi entropy fails to be convex along geodesics with respect to
this transportation metric \cite{ErMa14}.
A weaker notion than geodesic convexity (called convex entropy decay), which is
strongly related to the Bakry-Emery method, was introduced by Maas and Matthes 
\cite{MaMa15} to prove exponential decay rates for finite-volume discretizations
of the quantum drift-diffusion equation. Its gradient flow formulation
is based on the Fisher information and the logarithmic mean.

In this, paper, we propose a new mean function by composing the logarithmic mean
with a nonlinear function (coming from the diffusivity), 
which is suitable for finite-difference
discretizations of the nonlinear Fokker-Planck equation
\begin{equation}\label{1.fp}
  \pa_t\rho = \pa_x\big(\pa_x\phi(\rho)-\phi(\rho)\pa_x V\big), \quad x\in(0,1),\ t>0,
\end{equation}
supplemented with no-flux boundary conditions and an initial condition.
Here, $\phi:[0,\infty)\to[0,\infty)$ is a continuous function
and $V(x)$ is a confinement potential. 
An example is $\phi(\rho)=\rho^\alpha$ with $\alpha>0$ and $V(x)=\gamma |x|^2/2$ 
with $\gamma\ge 0$. A computation shows that the entropy 
$$
  \en_c(\rho) = \int_0^1 \big(f(\rho) + \rho V(x)\big)dx, \quad\mbox{where } 
	f'(s) = \log\phi(s), 
$$
is nonincreasing along (smooth) solutions to \eqref{1.fp}.
The displacement convexity of equations related to \eqref{1.fp} was
analyzed in \cite{CLSS10}.
Our aim is to show that a discrete version of the entropy $\en_c$
is displacement convex along semidiscrete solutions associated to \eqref{1.fp}. 

For the discretization of \eqref{1.fp}, let $n\in\N$, $h=1/n>0$, and $x_i=ih$,
$i=0,\ldots,n$. Let $\rho_i(t)$ approximate the solution $\rho(x_i,t)$ and 
$w_i$ approximate the function $w(x_i)=e^{-V(x_i)}$. 
Writing \eqref{1.fp} in the form
$$
  \pa_t \rho = \diver\bigg(\phi(\rho)\na\log\frac{\phi(\rho)}{w}\bigg),
$$
a corresponding finite-difference scheme reads as
\begin{equation}\label{1.dfp}
  \pa_t \rho_i = 
	\frac{\kappa_i\Lambda_i}{h^2}\bigg(\log\frac{\phi(\rho_{i+1})}{w_{i+1}}
	-\log\frac{\phi(\rho_i)}{w_i}\bigg)
	- \frac{\kappa_{i-1}\Lambda_{i-1}}{h^2}\bigg(\log\frac{\phi(\rho_{i})}{w_{i}}
	-\log\frac{\phi(\rho_{i-1})}{w_{i-1}}\bigg),
\end{equation}
where $h>0$ is the space size 
and $\kappa_i\Lambda_i$ is an approximation of $\phi(\rho)$ in $[x_i,x_{i+1}]$.
Our idea is to employ the {\em modified} logarithmic mean
\begin{equation}\label{1.La}
  \Lambda_i = \frac{u_i-u_{i+1}}{\log u_i-\log u_{i+1}}
	= \frac{\phi(\rho_i)/w_i-\phi(\rho_{i+1})/w_{i+1}}{\log(\phi(\rho_i)/w_i)
	- \log(\phi(\rho_{i+1})/w_{i+1})},
\end{equation}
and to set, as in \cite{Mie13}, $\kappa_i=\sqrt{w_iw_{i+1}}$. 
Since $\Lambda_i$ approximates
$u_i=\phi(\rho_i)/w_i$, it follows that 
$\kappa_i\Lambda_i$ approximates $\sqrt{w_{i+1}/w_i}\phi(\rho_i)$. 
Observe that with this choice, the numerical scheme reduces to
$$
  \pa_t\rho_i = \frac{\kappa_i}{h^2}(u_{i+1}-u_i) 
	- \frac{\kappa_{i-1}}{h^2}(u_i-u_{i-1}), \quad u_i=\frac{\phi(\rho_i)}{w_i},
$$
which approximates \eqref{1.fp} written 
in the form $\pa_t\rho=\pa_x(w\pa_x(\phi(\rho)/w))$.

The main result of the paper is as follows. If $\phi$ is invertible
and $\phi'\circ\phi^{-1}$ is nonincreasing (an example is $\phi(s)=s^\alpha$
with $0<\alpha<1$), then the discrete entropy 
\begin{equation}\label{1.en}
  \en(\rho) 
	= \sum_{i=0}^n\big(f(\rho_i) + \rho_i V(x_i)\big), \quad\mbox{where } 
	f'(s) = \log\phi(s), 
\end{equation}
is displacement $\lambda_h$-convex with respect to the 
nonlocal transportation metric induced by \eqref{1.La}, where
$$
  \lambda_h = \gamma\bigg(\frac{2}{\gamma h^2}(1-e^{-\gamma h^2/2})
	\min_{i=0,\ldots,n}\phi'(\rho_i) 
	- 2\cosh(\gamma h)\max_{i=0,\ldots,n}|\na_h\phi'(\rho_i)|\bigg) \in\R,
$$
and $\na_h\phi'(\rho_i)=h^{-1}(\phi'(\rho_{i+1})-\phi'(\rho_i))$; see
Theorem \ref{thm.main}. If the minimum of $\phi'(\rho_i)$ is positive and 
the maximum of $|\na_h\phi'(\rho_i)|$ is sufficiently small,
then $\lambda_h$ is positive. Such bounds in terms of the initial data can be
shown at least for the case $V=0$; see Corollary \ref{coro2}.
We expect that exponential convergence to the steady state holds for
sufficiently small $h>0$ (and $V\neq 0$), but we are unable to prove it.
Our result is consistent with that one in \cite{Mie13}: If $\phi(s)=s$ is
linear (and $V\neq 0$), $\lambda_h\to \gamma$ as $h\to 0$, and the constant is 
asymptotically sharp. 

The paper is organized as follows. In Section \ref{sec.dc}, we introduce the
mathematical setting and give the definition of displacement $\lambda$-convexity.
We show that displacement $\lambda$-convexity follows if a certain matrix
is positive semidefinite, slightly generalizing Proposition 2.1 in \cite{Mie13}.
As a warm-up, we consider in Section \ref{sec.heat} the semidiscrete heat
equation and prove that the entropy $\en(\rho)=\sum_{i=0}^n f(\rho_i)$
is displacement convex if $f(s)=s(\log s-1)$ or $f(s)=s^\alpha$ for 
$1<\alpha\le 2$; see Theorem \ref{thm.heat}. This result
is a reformulation of Theorem 5 in \cite{JuYu15}, but our proof is very
simple. Section \ref{sec.nfp} is concerned with the proof of displacement
$\lambda$-convexity of \eqref{1.en} and contains our main result. 
Some properties of mean functions
are recalled in Appendix \ref{sec.mean}, and a priori estimates of solutions
to \eqref{1.dfp} with $V=0$ are proved in Appendix \ref{sec.apriori}.


\section{Displacement convexity}\label{sec.dc}

In this section, we specify our setting and give the definition of
displacement convexity. Let $n\in\N$ and introduce the finite state space
$$
  X_n = \bigg\{\rho=(\rho_0,\ldots,\rho_n)\in\R^{n+1}:\rho_0,\ldots,\rho_n>0,\ 
	\sum_{i=0}^n\rho_i=1\bigg\}.
$$
This space can be identified with the space of probability measures on a
$(n+1)$-point set. We define the inner product 
$\langle\rho,\rho^*\rangle=\sum_{i=1}^n\rho_i\rho^*_i$
for $\rho$, $\rho^*\in X_n$. Let a matrix $Q=(Q_{ij})\in\R^{(n+1)\times(n+1)}$
be given such that
$$
  Q_{ij}\ge 0\mbox{ for }i\neq j, \quad \sum_{i=0}^n Q_{ij}=0\mbox{ for }
	j=1,\ldots,n.
$$
The value $Q_{ij}$ is the rate of a particle moving from state $j$ to $i$.
We assume that there exists a unique vector $w\in X_n$ such that the
detailed balance condition
$$
  Q_{ij}w_j = Q_{ji}w_i\quad\mbox{for all }i,j=0,\ldots,n
$$
is satisfied. Summing this condition for fixed $i$ over $j=0,\ldots,n$,
we see that $Qw=0$. Note that in Markov chain theory, the detailed balance condition
is usually formulated for the transposed matrix $Q^\top$. 

Our aim is to show convexity properties of the entropy 
along solutions $t\mapsto \rho(t)$ to ODE systems of the type
\begin{equation}\label{2.ode}
  \pa_t\rho = Q\phi(\rho), \quad t>0,
\end{equation}
where $\phi$ is some smooth function. This equation can be formulated as
a gradient flow. Indeed, given a (smooth) function $f:[0,\infty)\to\R$, we define the entropy $\en:X_n\to\R$,
\begin{equation}\label{2.en}
  \en(\rho) = \sum_{i=0}^n f_i(\rho_i),   \quad\mbox{where }  
	f_i'(s) = f'\bigg(\frac{\phi(s)}{w_i}\bigg), 
\end{equation}
and the Onsager operator $K:X_n\to\R^{(n+1)\times(n+1)}$,
\begin{equation}\label{2.K}
  K(\rho) = \frac12\sum_{i,j=0}^n Q_{ij}w_j
	\Lambda^f\bigg(\frac{\phi(\rho_i)}{w_i},\frac{\phi(\rho_{i+1})}{w_i}\bigg)
	(e_i-e_j)\otimes(e_i-e_j),
\end{equation}
where $e_i=(\delta_{i0},\ldots,\delta_{in})^\top\in\R^{n+1}$ is the $i$th unit
vector and ``$\otimes$'' is the tensor product. 
By detailed balance and $Q_{ij}w_j\ge 0$ for $i\neq j$, it follows that 
$K(\rho)$ is symmetric and positive semidefinite.
With these definitions, we can formulate \eqref{2.ode} as a gradient system
in the sense that it can be rewritten as
\begin{equation}\label{eq.gf}
  \pa_t\rho = -K(\rho)D\en(\rho),
\end{equation}
where $D\en(\rho)=(f_0'(\rho_0),\ldots,f_n'(\rho_n))$.

The space $X_n$ is endowed with the nonlocal transportation distance
\begin{equation}\label{2.W}
  \W(\rho_0,\rho_1)^2 = \inf_{(\rho,\psi)\in E(\rho_0,\rho_1)}
	\int_0^1\big\langle K(\rho(t)),\psi(t),\psi(t)\big\rangle dt,
\end{equation}
where $E(\rho_0,\rho_1)$ is the set of pairs $(\rho(t),\psi(t))$, $t\in[0,1]$,
such that
\begin{align*}
  & \rho\in C^1([0,1];X_n), \ \psi:[0,1]\to\R^{n+1}\mbox{ is measurable}, \\
	& \mbox{for all }i=0,\ldots,n,\ t\in[0,1]: \ \pa_t\rho(t) = -K(\rho)\psi(t),
	\quad \rho(0)=\rho_0, \ \rho(1)=\rho_1.
\end{align*}
It is well known that the function ${\mathcal W}$ is a 
pseudo-metric on $X_n$ \cite[Theorem~1.1]{Maa11} and the pair $(X_n,{\mathcal W})$ 
defines a geodesic space \cite[Prop.~2.3]{ErMa14}, i.e., for all
$\rho_0$, $\rho_1\in X_n$, there exists at least one curve
$\rho:[0,1]\to X_n$, $t\mapsto\rho(t)$, such that $\rho(0)=\rho_0$,
$\rho(1)=\rho_1$, and $\W(\rho(s),\rho(t))=|s-t|\W(\rho_0,\rho_1)$ 
for all $s$, $t\in[0,1]$. Such a curve
is called a constant speed geodesics between $\rho_0$ and
$\rho_1$. If the pair $(\rho,\psi)\in E(\rho_0,\rho_1)$ 
attains the infimum in \eqref{2.W},
then it satisfies the geodesic equations \cite[Prop.~2.5]{ErMa14}
\begin{equation}\label{2.ge}
  \left\{    \begin{array}{l}
   \pa_t\rho = K(\rho)\psi,\\ 
   \pa_t\psi = -\frac12\langle DK(\rho)[\,\cdot\,]\psi,\psi\rangle
  \end{array}  \right. , \quad t>0,
\end{equation}
where the vector $b=\langle DK(\rho)[\,\cdot\,]\psi,\psi\rangle$ is defined
by $\langle b,v\rangle = \langle DK(\rho)[v]\psi,\psi\rangle$ for $v\in X_n$.

\begin{definition}[Displacement convexity]\label{def.dc}
Let $\lambda\in\R$.
We say that a functional $\E:X_n\to\R\cup\{+\infty\}$ is displacement
$\lambda$-convex on $X_n$ with respect to the metric ${\mathcal W}$
if for any constant speed geodesic curve $\rho:[0,1]\to X_n$,
$$
  \E(\rho(t)) \le (1-t)\E(\rho(0)) + t\E(\rho(1))
	- \frac{\lambda}{2} t(1-t){\mathcal W}(\rho(0),\rho(1))^2, \quad t\in[0,1].
$$
If $\lambda=0$, $\E$ is simply called displacement convex.
Moreover, if $t\mapsto\E(\rho(t))$ is twice differentiable, 
$\E$ is displacement $\lambda$-convex if and only if
$$
  \frac{d^2}{dt^2}\E(\rho(t))\ge \lambda{\mathcal W}(\rho(0),\rho(1))^2,
	\quad t\in[0,1].
$$
\end{definition}

We show that displacement $\lambda$-convexity of $\en$ is guaranteed if
a certain matrix is positive semidefinite. This result (slightly) generalizes
Proposition 2.1 in \cite{Mie13}.

\begin{proposition}\label{prop}
The entropy $\en$, defined in \eqref{2.en}, is displacement $\lambda$-convex
for some $\lambda\in\R$ if for any $\rho\in X_n$, 
\begin{equation}\label{2.MK}
  M(\rho)\ge \lambda K(\rho),
\end{equation}
i.e.\ if $M(\rho)-\lambda K(\rho)$ is positive semidefinite, where
\begin{equation}\label{2.M}
  M(\rho)=\frac12(DK(\rho)[Q\phi(\rho)] - Q\Phi'(\rho)K(\rho) 
  - K(\rho)\Phi'(\rho)Q^\top)
\end{equation}
and $\Phi'(\rho)=\diag(\phi'(\rho_1),\ldots,\phi'(\rho_n))$.
\end{proposition}

\begin{proof}
Let $\rho_0$, $\rho_1\in X_n$ and let $\rho:[0,1]\to X_n$
be a geodesic curve with $(\rho,\psi)\in E(\rho_0,\rho_1)$.
Then $(\rho,\psi)$ satisfies the geodesic equations \eqref{2.ge}, implying that
$$
  \frac{d}{dt}\en(\rho) = \langle D\en(\rho),\pa_t\rho\rangle
	= \langle D\en(\rho),K(\rho)\psi\rangle.
$$
Differentiating a second time and using the symmetry of $K(\rho)$ and
$DK(\rho)[\pa_t\rho]$, we find that
\begin{align*}
  \frac{d^2}{dt^2}\en(\rho) 
	&= \langle D^2\en(\rho)\pa_t\rho,K(\rho)\psi\rangle
	+ \langle D\en(\rho),DK(\rho)[\pa_t\rho]\psi\rangle
	+ \langle D\en(\rho),K(\rho)\pa_t\psi\rangle \\
	&= \langle K(\rho)D^2\en(\rho)\pa_t\rho,\psi\rangle
	+ \langle DK(\rho)[\pa_t\rho]D\en(\rho),\psi\rangle
	+ \langle K(\rho)D\en(\rho),\pa_t\psi\rangle.
\end{align*}
Inserting the geodesic equations \eqref{2.ge} yields
\begin{align}
  \frac{d^2}{dt^2}\en(\rho) 
	&= \big\langle K(\rho)D^2\en(\rho)K(\rho)\psi + DK(\rho)[K(\rho)\psi]D\en(\rho),
	\psi\big\rangle \nonumber \\
	&\phantom{xx}{}
	- \frac12\big\langle DK(\rho)[K(\rho)D\en(\rho)]\psi,\psi\big\rangle. \label{2.aux}
\end{align}
We differentiate $K(\rho)D\en(\rho) = -Q\phi(\rho)$ with respect to $\rho$:
$$
  K(\rho)D^2\en(\rho) + DK(\rho)[\,\cdot\,]D\en(\rho) = -Q\Phi'(\rho).
$$
Thus, we can replace the first bracket on the right-hand side of \eqref{2.aux}
by $-Q\Phi'(\rho)K(\rho)\psi$:
\begin{align}
  \frac{d^2}{dt^2}\en(\rho) 
	&= -\langle Q\Phi'(\rho)K(\rho)\psi,\psi\rangle 
	+ \frac12\big\langle DK(\rho)[Q\phi(\rho)]\psi,\psi\big\rangle \nonumber \\
  &= \frac12\big\langle\big(DK(\rho)[Q\phi(\rho)]
	- Q\Phi'(\rho)K(\rho) - K(\rho)\Phi'(\rho)Q^\top\big)\psi,\psi\big\rangle.
	\label{2.d2F}
\end{align}
We infer from \eqref{2.MK} that
$$
  \frac{d^2}{dt^2}\en(\rho) \ge \lambda\langle K(\rho)\psi,\psi\rangle
$$
for all geodesic curves $\rho$ and vector fields $\psi$ such that
$(\rho,\psi)\in E(\rho_0,\rho_1)$. Consequently,
$$
  \frac{d^2}{dt^2}\en(\rho(t))\ge \lambda{\mathcal W}(\rho_0,\rho_1)^2,
	\quad t\in[0,1],
$$
and by Definition \ref{def.dc}, $\en$ is displacement $\lambda$-convex.
\end{proof}


\section{Semidiscrete heat equation}\label{sec.heat}

As a warm-up, we consider the semidiscrete heat equation
\begin{equation}\label{3.eq}
  \pa_t \rho_i = h^{-2}(\rho_{i-1} - 2\rho_i + \rho_{i+1}),
	\quad i=0,\ldots,n,\ t>0,
\end{equation}
where $n\in\N$ and $h=1/n>0$.
The no-flux boundary conditions are realized by setting $\rho_{-1}=\rho_0$ and
$\rho_{n+1}=\rho_n$. We write $\rho=(\rho_0,\ldots,\rho_n)$.
Equation \eqref{3.eq} can be written as \eqref{2.ode} by setting $\phi(s)=s$ and
$Q=-G^\top G$ with the discrete gradient $G\in\R^{(n+1)\times(n+1)}$,
$G_{ij}=h^{-1}(\delta_{ij}-\delta_{i+1,j})$. By slightly abusing the notation,
we set $w_i=1$ for $i=0,\ldots,n$ and note that for a function $f:[0,\infty)\to\R$ , the corresponding entropy given in \eqref{2.en} reduces to
\begin{equation}\label{3.en}
\en(\rho) = \sum_{i=0}^n f(\rho_i).
\end{equation}
Then, for the respective Onsager operator given in \eqref{2.K} with the mean function $\Lambda^f$, we claim that the entropy $\en$ is displacement convex, under suitable conditions on $f$. 

\begin{theorem}\label{thm.heat}
Let $f$ be such that $\Lambda^f$, defined in \eqref{1.Laf}, is concave in both variables. Then 
the entropy \eqref{3.en} is displacement convex with respect to the metric
\eqref{2.W} induced by $\Lambda^f$.
\end{theorem}

If $f(s)=s(\log s-1)$ or $f(s)=s^\alpha$ for $1<\alpha\le 2$, $\Lambda$ is
concave in both variables (see Lemma \ref{lem.mean}), thus fulfilling the
assumption of the theorem.

\begin{proof}
We formulate $Q\rho=-G^\top G\rho=-G^\top L(\rho)G f'(\rho)$,
where $L(\rho)=\diag(\Lambda^f(\rho_i,\rho_{i+1}))_{i=0}^n$
and $f'(\rho)=(f'(\rho_i))_{i=0}^n$. Then, setting
$K(\rho)=G^\top L(\rho)G$, we can write \eqref{3.eq} as the gradient system
$$
  \pa_t\rho = Q\rho = -K(\rho)D\en(\rho),
$$
where we identify $D\en(\rho)$ with $f'(\rho)$. Thus, by 
Proposition \ref{prop}, it is sufficient to show that the matrix $M(\rho)$,
defined in \eqref{2.M}, is positive semidefinite. In fact, because of the
special structure of $K(\rho)$, we can simplify this condition. 
Let $\psi\in\R^{n+1}$. Then, using $DK(\rho)[\,\cdot\,]=G^\top DL(\rho)[\,\cdot\,]G$
and $Q=-G^\top G$,
\begin{align*}
  \langle M(\rho)\psi,\psi\rangle
	&= \frac12\Big\langle \big(DK(\rho)[Q\rho] - QK(\rho) - K(\rho)Q^\top\big)\psi,\psi
	\Big\rangle \\
	&= \frac12\Big\langle G^\top\big(DL(\rho)[Q\rho]G + GG^\top L(\rho)G
	+ L(\rho)GG^\top G\big)\psi,\psi\Big\rangle \\
	&=  \frac12\Big\langle \big(DL(\rho)[Q\rho] + GG^\top L(\rho) 
	+ L(\rho)GG^\top \big)G\psi,G\psi\Big\rangle.
\end{align*}
Hence, it is sufficient to show that
$$
  \widetilde M := -DL(\rho)[G^\top G\rho] + GG^\top L(\rho) + L(\rho)GG^\top
$$
is positive semidefinite. 

We show this claim by verifying that $\widetilde M$ is
diagonally dominant. To this end, we observe that $\widetilde M$ is a symmetric
tridiagonal matrix with entries
$$
  \widetilde M = \frac{1}{h^2}\begin{pmatrix} 
	a_1 & b_1 & 0   & \cdots & 0 \\
	b_1 & a_2 & b_2 & \ddots & \vdots \\
	0   & b_2 & \ddots &     & 0 \\
	\vdots & \ddots &  & a_{n-1} & b_{n-1} \\
	0   & \cdots & 0 & b_{n-1}   & a_n 
	\end{pmatrix},
$$
where the coefficients are given by
\begin{align*}
  a_i &= 4\Lambda^f(\rho_i,\rho_{i+1}) 
	- \pa_1\Lambda^f(\rho_i,\rho_{i+1})(2\rho_i-\rho_{i-1}-\rho_{i+1})\\
	&\phantom{xx}{}
  - \pa_2\Lambda^f(\rho_i,\rho_{i+1})(2\rho_{i+1}-\rho_{i}-\rho_{i+2}), \\
	b_i &= -\big(\Lambda^f(\rho_i,\rho_{i+1})+\Lambda^f(\rho_{i+1},\rho_{i+2})\big) 
	\le 0,
	\quad i=1,\ldots,n.
\end{align*}
We also set $b_0=-\Lambda^f(\rho_0,\rho_1)-\Lambda^f(\rho_1,\rho_2)\le 0$.

The matrix $\widetilde M$ is diagonally dominant if
\begin{align}
  & a_1+b_1 \ge 0, \quad a_n + b_{n-1} \ge 0, \label{3.dd1} \\
	& a_i+b_{i-1}+b_i \ge 0\quad\mbox{for }i=1,\ldots,n. \label{3.dd2} 
\end{align}
The first two conditions \eqref{3.dd1} follow from \eqref{3.dd2} for $i=1$ and
$i=n$, since
$a_1+b_1=(a_1+b_0+b_1)-b_0\ge a_1+b_0+b_1\ge 0$ and
$a_n+b_{n-1}=(a_n+b_{n-1}+b_n)-b_n\ge a_n+b_{n-1}+b_n\ge 0$.
Thus, it remains to prove \eqref{3.dd2}. We compute
\begin{align*}
  a_i+{} & b_{i-1}+b_i
	= 2\Lambda^f(\rho_i,\rho_{i+1}) - \Lambda^f(\rho_{i+1},\rho_{i+2})
	- \Lambda^f(\rho_{i-1},\rho_{i}) \\
	&{}- \pa_1\Lambda^f(\rho_i,\rho_{i+1})
	\big(2\rho_i-\rho_{i-1}-\rho_{i+1}\big) 
	- \pa_2\Lambda^f(\rho_i,\rho_{i+1})\big(2\rho_{i+1}-\rho_{i}-\rho_{i+2}\big).
\end{align*}
Since $\Lambda^f$ is assumed to be concave, we may apply 
Lemma \ref{lem.second}, which shows that this expression is nonnegative,
and hence, $\widetilde M$ is positive semidefinite.
\end{proof}

For nonlinear functions $\phi$ and nonconstant steady states $(w_i)$, 
the proof of nonnegativity of $a_i+b_{i-1}+b_i$ is,
unfortunately, not as simple as above, and we need more properties of the mean 
function. It turns out that the logarithmic mean satisfies these properties.
Such a situation is considered in the next section.


\section{Semidiscrete nonlinear Fokker-Planck equations}\label{sec.nfp}

We discretize the nonlinear Fokker-Planck equation
$$
  \pa_t\rho = \diver(\na\phi(\rho)-\phi(\rho)\na V) 
	= \diver\bigg(w\na\frac{\phi(\rho)}{w}\bigg),
$$
where $w(x)=e^{-V(x)}$. We choose the quadratic potential
$V(x)=\gamma|x|^2/2$ with $\gamma>0$ but other choices are possible.
Let $n\in\N$, $h=1/n>0$, and $x_i=ih$.
Approximating $\rho(x_i,t)$ by $\rho_i(t)$, $w(x_i)$ by $w_i$ and setting
$u_i=\phi(\rho_i)/w_i$, the numerical scheme reads as
\begin{equation}\label{4.ns}
  \pa_t\rho_i = h^{-2}\kappa_i(u_{i+1}-u_i) - h^{-2}\kappa_i(u_i-u_{i-1}),
\end{equation}
where $\kappa_i=\sqrt{w_iw_{i+1}}$ approximates $w(x_{i+1/2})$.
The no-flux boundary conditions are realized by $u_{-1}=u_0$ and $u_{n+1}=u_n$.
Setting $Q=G^\top\diag(\kappa_i)G\diag(w_i^{-1})$ and, slightly abusing the
notation, $\rho=(\rho_0,\ldots,\rho_n)$, we see that the scheme can be formulated
as $\pa_t\rho=Q\phi(\rho)$, and thus, the framework of Section \ref{sec.dc} 
applies. Hence, \eqref{4.ns} can be written as the gradient system
$$
  \pa_t\rho = -K(\rho)\log u, \quad K(\rho)=G^\top L(\rho)G,
$$
where $\log u=(\log u_i)_{i=0}^n$,
$$
  L(\rho) = \diag\big(\kappa_i\Lambda(u_i,u_{i+1})\big)_{i=0}^n, \quad
	u_i = \frac{\phi(\rho_i)}{w_i},
$$
and $\Lambda$ is the logarithmic mean. The above system can be written as in 
\eqref{eq.gf} by chosing $f(s)=s(\log s-1)$, and therefore, by \eqref{2.en}, 
the entropy reads as
$$
  \en(\rho) = \sum_{i=0}^n\bigg(f(\rho_i) + \frac{\gamma}{2}x_i^2\rho_i\bigg),
$$
since
\[f_i'(s) = f'\left(\frac{\phi(s)}{w_i}\right)=\log \phi(s) -\log w_i = f'(s)+\frac{\gamma}{2}x_i^2, \qquad i=0,\ldots n. \]
Thus, $D\en(\rho)=\log u$ and, for the nonlocal transportation metric $\W$ defined in \eqref{2.W}, we have the following result.

\begin{theorem}\label{thm.main}
Let $\phi$ be invertible, $\phi'\circ\phi^{-1}$ be nonincreasing, and $\gamma>0$. 
Then the entropy $\en$ is displacement $\lambda_h$-convex with respect to $\W$,
where
$$
  \lambda_h = \gamma\bigg(\frac{2}{\gamma h^2}(1-e^{-\gamma h^2/2})
	\min_{i=0,\ldots,n}\phi'(\rho_i) 
	- 2\cosh(\gamma h)\max_{i=0,\ldots,n}|\na_h\phi'(\rho_i)|\bigg) \in\R.
$$
If $\phi(s)=s$, we have $\lambda_h=(2/h^{2})(1-e^{-\gamma h^2/2})\to \gamma$ 
as $h\to 0$.
\end{theorem}

From numerical analysis, we expect that $\min_{i=0,\ldots,n}\phi'(\rho_i)$ and
$\max_{i=0,\ldots,n}|\na_h\phi'(\rho_i)|$ are independent of $h$ and bounded
only by discrete norms of $\rho(0)$. In Appendix \ref{sec.apriori}, we provide
such estimates for the case $V=0$. These estimates show that $\lambda_h$ is positive
if $\max_i|\na_h\rho_i(0)|$ is sufficiently small. The function
$\phi(s)=s^\alpha$ satisfies the assumptions of the theorem if $0<\alpha\le 1$.
In the linear case $\phi(s)=s$, we recover essentially the result of \cite{Mie13}.

\begin{proof}
According to Proposition \ref{prop}, it is sufficient to show that the matrix
$M(\rho)-\lambda_h K(\rho)$ is positive semidefinite. 
The derivative of $K(\rho)$ becomes
$DK(\rho[\,\cdot\,]=G^\top DL(\rho)[\,\cdot\,]G$ and
$$
  (DL(\rho)[\xi])_i = \kappa_i\pa_1\Lambda(u_i,u_{i+1})\frac{\phi'(\rho_i)}{w_i}\xi_i
	+ \kappa_i\pa_2\Lambda(u_i,u_{i+1})\frac{\phi'(\rho_{i+1})}{w_{i+1}}\xi_{i+1}
$$
for $i=0,\ldots,n$ and $\xi\in\R^{n+1}$. Therefore, for $\psi\in\R^{n+1}$,
\begin{align*}
  \langle M(\rho)\psi,\psi\rangle
	&= \frac12\big\langle G^\top\big\{DL(\rho)[Q\phi(\rho)]G
	+ Q\Phi'(\rho)G^\top L(\rho)G + G^\top L(\rho)G\Phi'(\rho)Q^\top\big\}\psi,
	\psi\big\rangle \\
	&= \frac12\langle \widetilde MG\psi,G\psi\rangle,
\end{align*}
where
\begin{align*}
  \widetilde M 
	&= DL(\rho)[Q\phi(\rho)]
	+ \diag(\kappa_i)G\diag(w_i^{-1})\Phi'(\rho)G^\top L(\rho) \\
	&\phantom{xx}{}
	+ L(\rho)G\Phi'(\rho)\diag(w_i^{-1})G^\top\diag(\kappa_i).
\end{align*}
This matrix is symmetric and tridiagonal with entries
$$
  \widetilde M = \frac{1}{h^2}\begin{pmatrix} 
	a_1 & b_1 & 0   & \cdots & 0 \\
	b_1 & a_2 & b_2 & \ddots & \vdots \\
	0   & b_2 & \ddots &     & 0 \\
	\vdots & \ddots &  & a_{n-1} & b_{n-1} \\
	0   & \cdots & 0 & b_{n-1}   & a_n 
	\end{pmatrix},
$$
where the coefficients are given by
\begin{align*}
  a_i &= 2\kappa_i\Lambda_i\bigg(\frac{\phi'(\rho_i)}{w_i}
	+\frac{\phi'(\rho_{i+1})}{w_i}\bigg) - \kappa_i\frac{\phi'(\rho_i)}{w_i}
	\pa_1\Lambda_i\big(\kappa_{i-1}(u_i-u_{i-1}) + \kappa_i(u_i-u_{i+1})\big) \\
	&\phantom{xx}{}
	- \kappa_i\frac{\phi'(\rho_{i+1})}{w_{i+1}}\pa_2\Lambda_i\big(\kappa_i(u_{i+1}-u_i)
	+ \kappa_{i+1}(u_{i+1}-u_{i+2})\big), \\
  b_i &= -\kappa_i\kappa_{i+1}\frac{\phi�(\rho_{i+1})}{w_{i+1}}
	(\Lambda_i+\Lambda_{i+1}) \le 0,
\end{align*}
and we abbreviated
$$
  \Lambda_i:=\Lambda(u_i,u_{i+1}), \quad
	\pa_j\Lambda_i:=\pa_j\Lambda(u_i,u_{i+1}), \quad j=1,2.
$$

We show now that $\widetilde M-\lambda_h L(\rho)$ is diagonally dominant for
some $\lambda\in\R$. For this, we introduce further abbreviations:
$$
  \alpha_i = \kappa_i\frac{\phi'(\rho_i)}{w_i}, \quad
	\beta_i = \kappa_i\frac{\phi'(\rho_{i+1})}{w_{i+1}}.
$$
Since $\kappa_i\alpha_{i+1}=\kappa_{i+1}\beta_i$, we compute
\begin{align}
  a_i+b_{i-1}+b_i
	&= 2\kappa_i\Lambda_i(\alpha_i+\beta_i) 
	- \kappa_i\beta_{i-1}(\Lambda_{i-1}+\Lambda_i)
	- \kappa_i\alpha_{i+1}(\Lambda_i+\Lambda_{i+1}) \nonumber \\
	&\phantom{xx}{}- \kappa_{i}\alpha_i\pa_1\Lambda_i(u_i-u_{i+1})
	- \kappa_i\beta_i\pa_2\Lambda_i(u_{i+1}-u_i) \nonumber \\
	&\phantom{xx}{}- \kappa_{i-1}\alpha_i\pa_1\Lambda_i(u_i-u_{i-1})
	- \kappa_{i+1}\beta_i\pa_2\Lambda_i(u_{i+1}-u_{i+2}) \nonumber \\
	&= \kappa_i\Lambda_i(2\alpha_i+2\beta_i-\beta_{i-1}-\alpha_{i+1}) \nonumber \\
	&\phantom{xx}{}- \kappa_i\alpha_i\pa_1\Lambda_i(u_i-u_{i+1})
	- \kappa_i\beta_i\pa_2\Lambda_i(u_{i+1}-u_i) \nonumber \\
	&\phantom{xx}{}- \kappa_i\beta_{i-1}(\Lambda_{i-1}-\pa_1\Lambda_i u_{i-1})
	- \kappa_i\alpha_{i+1}(\Lambda_{i+1}-\pa_2\Lambda_iu_{i+2}) \nonumber \\
	&\phantom{xx}{}- \kappa_{i-1}\alpha_i\pa_1\Lambda_i u_i
	- \kappa_{i+1}\beta_i\pa_2\Lambda_iu_{i+1} \nonumber \\
  &= I_1+\cdots+I_7. \label{4.aux1}
\end{align}
We estimate these expressions term by term. Using property (ii) of
Lemma \ref{lem.Lambda}, we find that
$$
  I_2 = -\kappa_i\alpha_i\Lambda_i + \kappa_i\alpha_i\frac{\Lambda_i^2}{u_i}, \quad
	I_3 = -\kappa_i\beta_i\Lambda_i + \kappa_i\beta_i\frac{\Lambda_i^2}{u_{i+1}}.
$$
The first terms on the right-hand sides cancel with some terms in $I_1$.
By property (iv) of Lemma \ref{lem.Lambda}, it follows that
\begin{align*}
  I_4 &\ge -\kappa_i\beta_{i-1}\max_{r\ge 0}
	\big(\Lambda(r,u_{i})-\pa_1\Lambda(u_i,u_{i+1})r\big) 
	= -\kappa_i\beta_{i-1}u_{i}\pa_2\Lambda(u_i,u_{i+1}) \\
	&= -\kappa_i\beta_{i-1}u_i\pa_2\Lambda_i, \\
	I_5 &\ge -\kappa_i\alpha_{i+1}\max_{r\ge 0}
	\big(\Lambda(u_{i+1},r)-\pa_2\Lambda(u_i,u_{i+1})r\big) \\
	&= -\kappa_i\alpha_{i+1}\max_{r\ge 0}
	(\Lambda(r,u_{i+1})-\pa_1\Lambda(u_{i+1},u_{i})r) 
	= -\kappa_i\alpha_{i+1}u_{i+1}\pa_2\Lambda(u_{i+1},u_i) \\
	&= -\kappa_i\alpha_{i+1}u_{i+1}\pa_1\Lambda(u_{i},u_{i+1})
	= -\kappa_i\alpha_{i+1}u_{i+1}\pa_1\Lambda_i.
\end{align*}
Finally, because of $\kappa_i\alpha_{i+1}=\kappa_{i+1}\beta_i$,
$$
  I_6 = -\kappa_i\beta_{i-1}\pa_1\Lambda_iu_i, \quad
	I_7 = -\kappa_i\alpha_{i+1}\pa_2\Lambda_i u_{i+1}.
$$

Inserting these computations into \eqref{4.aux1}, we arrive at
\begin{align*}
  a_i+b_{i-1}+b_i 
	&\ge \kappa_i\Lambda_i(\alpha_i+\beta_i-\beta_{i-1}-\alpha_{i+1})
	+ \kappa_i\Lambda_i^2\bigg(\frac{\alpha_i}{u_i}+\frac{\beta_i}{u_{i+1}}\bigg) \\
	&\phantom{xx}{}- \kappa_i(\beta_{i-1}u_i+\alpha_{i+1}u_{i+1})
	(\pa_1\Lambda_i+\pa_2\Lambda_i).
\end{align*}
Employing property (iii) of Lemma \ref{lem.Lambda} in the last term, we obtain
\begin{align}
  a_i+b_{i-1}+b_i 
	&\ge \kappa_i\Lambda_i(\alpha_i+\beta_i-\beta_{i-1}-\alpha_{i+1})
	+ \kappa_i\Lambda_i^2\bigg(\frac{\alpha_i-\alpha_{i+1}}{u_i}
	+ \frac{\beta_i-\beta_{i-1}}{u_{i+1}}\bigg) \nonumber \\
	&= J_1 + J_2. \label{4.aux2}
\end{align}
The idea is to replace $\kappa_{i\pm 1}$ in $\beta_{i-1}$ and $\alpha_{i+1}$ 
by an expression involving only $\kappa_i$. 
By definition of $\alpha_i$ and $\beta_i$ and since 
\begin{align*}
  \frac{\kappa_{i+1}}{w_{i+1}} 
	&= \frac{\kappa_i}{w_i}\,\frac{\kappa_{i+1}}{\kappa_i}\,\frac{w_i}{w_{i+1}}
	= \frac{\kappa_i}{w_i}\,\frac{\sqrt{w_iw_{i+2}}}{w_{i+1}}
	= \frac{\kappa_i}{w_i}e^{-\gamma h^2/2}, \\
	\frac{\kappa_{i-1}}{w_i}
	&= \frac{\kappa_i}{w_{i+1}}\,\frac{\kappa_{i-1}}{\kappa_i}
	\,\frac{w_{i+1}}{w_i}
	= \frac{\kappa_i}{w_i}\,\frac{\sqrt{w_{i-1}w_{i+1}}}{w_{i}}
	= \frac{\kappa_i}{w_{i}}e^{-\gamma h^2/2},
\end{align*}
we find that
\begin{align*}
  J_1 &= \kappa_i\Lambda_i\bigg(\frac{\kappa_i}{w_i}\phi'(\rho_i) 
	- \frac{\kappa_{i+1}}{w_{i+1}}\phi'(\rho_{i+1})
	+ \frac{\kappa_i}{w_{i+1}}\phi'(\rho_{i+1})
	- \frac{\kappa_{i-1}}{w_i}\phi'(\rho_i)\bigg) \\
	&= \frac{\kappa_i^2}{w_i}\Lambda_i
	\big(\phi'(\rho_i)-e^{-\gamma h^2/2}\phi'(\rho_{i+1})\big)
	+ \frac{\kappa_i^2}{w_{i+1}}\Lambda_i\big(\phi'(\rho_{i+1})-e^{-\gamma h^2/2}
	\phi'(\rho_{i})\big).
\end{align*}
In the same way, since
$$
  \kappa_{i+1}\frac{w_i}{w_{i+1}} = \kappa_i\frac{\sqrt{w_iw_{i+2}}}{w_{i+1}}
	= \kappa_i e^{-\gamma h^2/2}, \quad
	\kappa_{i-1}\frac{w_{i+1}}{w_i} = \kappa_i\frac{\sqrt{w_{i-1}w_{i+1}}}{w_i}
	= \kappa_i e^{-\gamma h^2/2},
$$
we infer that
\begin{align*}
  J_2 &= \kappa_i\Lambda_i^2\bigg(\kappa_i\frac{\phi'(\rho_i)}{\phi(\rho_i)}
  - \kappa_{i+1}\frac{w_i}{w_{i+1}}\frac{\phi'(\rho_{i+1})}{\phi(\rho_i)}
  + \kappa_i\frac{\phi'(\rho_{i+1})}{\phi(\rho_{i+1})}
	- \kappa_{i-1}\frac{w_{i+1}}{w_i}\frac{\phi'(\rho_i)}{\phi(\rho_{i+1})}\bigg) \\
	&= \kappa_i^2\Lambda_i^2\bigg(
	\frac{\phi'(\rho_i)-e^{-\gamma h^2/2}\phi'(\rho_{i+1})}{\phi(\rho_i)}
	+ \frac{\phi'(\rho_{i+1})-e^{-\gamma h^2/2}\phi'(\rho_i)}{\phi(\rho_{i+1})}
  \bigg).
\end{align*}
Thus, \eqref{4.aux2} becomes
\begin{align}
  a_i+b_{i-1}+b_i 
	&\ge \kappa_i^2\Lambda_i\bigg(
	\frac{\phi'(\rho_i)-e^{-\gamma h^2/2}\phi'(\rho_{i+1})}{w_i}
	+ \frac{\phi'(\rho_{i+1})-e^{-\gamma h^2/2}\phi'(\rho_{i})}{w_{i+1}}\bigg) 
	\nonumber \\
	&\phantom{xx}{}	+ \kappa_i^2\Lambda_i^2
	\bigg(\frac{\phi'(\rho_i)-e^{-\gamma h^2/2}\phi'(\rho_{i+1})}{\phi(\rho_i)}
	+ \frac{\phi'(\rho_{i+1})-e^{-\gamma h^2/2}\phi'(\rho_{i})}{\phi(\rho_{i+1})}\bigg) 
	\nonumber \\
	&= \kappa_i^2\Lambda_i\big(\phi'(\rho_i)-\phi'(\rho_{i+1})\big)
	\bigg[\Lambda(u_i,u_{i+1})\bigg(\frac{1}{\phi(\rho_i)}-\frac{1}{\phi(\rho_{i+1})}
	\bigg) + \frac{1}{w_i} - \frac{1}{w_{i+1}}\bigg] \nonumber \\
	&\phantom{xx}{}
	+ \kappa_i^2\Lambda_i\big(1-e^{-\gamma h^2/2}\big)\bigg[\frac{\phi'(\rho_i)}{w_{i+1}}
	+ \frac{\phi'(\rho_{i+1})}{w_i} \nonumber \\
	&\phantom{xx}{}+ \Lambda(u_i,u_{i+1})
	\bigg(\frac{\phi'(\rho_i)/w_{i+1}}{u_{i+1}} + \frac{\phi'(\rho_{i+1})/w_i}{u_i}
	\bigg)\bigg] \nonumber \\
	&= K_1 + K_2. \label{4.aux3}
\end{align}
First, we estimate $K_2$ using property (v) of Lemma \ref{lem.Lambda}:
\begin{align*}
  K_2 &\ge 2\kappa_i^2\Lambda_i
	\big(1-e^{-\gamma h^2/2}\big)\bigg(\frac{\phi'(\rho_i)}{w_{i+1}}
	+ \frac{\phi'(\rho_{i+1})}{w_i} 
	+ 2\sqrt{\frac{\phi'(\rho_i)\phi'(\rho_{i+1})}{w_i w_{i+1}}}\bigg) \\
	&\ge 2\kappa_i\Lambda_i\big(1-e^{-\gamma h^2/2}\big)
	\sqrt{\phi'(\rho_i)\phi'(\rho_{i+1})}
	\ge 2\kappa_i\Lambda_i\big(1-e^{-\gamma h^2/2}\big)\min_{i=0,\ldots,n}\phi'(\rho_i).
\end{align*}
Since $\phi'\circ\phi^{-1}$ is nonincreasing, we have
$$
  \big(\phi'(\rho_i)-\phi'(\rho_{i+1})\big)
	\bigg(\frac{1}{\phi(\rho_i)}-\frac{1}{\phi(\rho_{i+1})}\bigg) \ge 0.
$$
Consequently, since $\Lambda(u_i,u_{i+1})\ge 0$ and $\sinh(s)\le s\cosh(s)$ 
for $s\ge 0$,
\begin{align*}
  K_1 &\ge \kappa_i^2\Lambda_i\big(\phi'(\rho_i)-\phi'(\rho_{i+1})\big)
	\bigg(\frac{1}{w_i} - \frac{1}{w_{i+1}}\bigg) \\
	&= \kappa_i\Lambda_i\big(\phi'(\rho_i)-\phi'(\rho_{i+1})\big)
	\bigg(\sqrt{\frac{w_{i+1}}{w_i}} - \sqrt{\frac{w_i}{w_{i+1}}}\bigg) \\
	&= -\kappa_i\Lambda_i\big(\phi'(\rho_i)-\phi'(\rho_{i+1})\big)
	\big(e^{\gamma(x_{i+1}^2-x_i^2)/4}-e^{-\gamma(x_{i+1}^2-x_i^2)/4}\big) \\
	&\ge -2\kappa_i\Lambda_i h\max_{i=0,\ldots,n}|\na_h\phi'(\rho_i)|
	\sinh\bigg(\frac{\gamma}{4}(2i+1)h^2\bigg) \\
	&\ge -2\kappa_i\Lambda_i h\max_{i=0,\ldots,n}|\na_h\phi'(\rho_i)|
	\bigg(\frac{\gamma}{4}(2i+1)h^2\bigg)\cosh\bigg(\frac{\gamma}{4}(2i+1)h^2\bigg) \\
	&\ge -2\kappa_i\Lambda_i h^2\max_{i=0,\ldots,n}|\na_h\phi'(\rho_i)|
	\gamma \cosh(\gamma h),
\end{align*}
where we recall that $|\na_h\phi'(\rho_i)|:=h^{-1}|\phi'(\rho_i)-\phi'(\rho_{i+1})|$
and we used $h\le 1$.
Then \eqref{4.aux3} yields
\begin{align*}
  h^{-2} (a_i+b_{i-1}+b_i) &\ge
	\gamma\kappa_i\Lambda_i
	\bigg(\frac{2}{\gamma h^2}(1-e^{-\gamma h^2/2})\min_{i=0,\ldots,n}\phi'(\rho_i) \\
	&\phantom{xx}{}
	- 2\cosh(\gamma h)\max_{i=0,\ldots,n}|\na_h\phi'(\rho_i)|\bigg) \\
	&= \lambda_h\kappa_i\Lambda_i. 
\end{align*}
This proves that $\widetilde M-\lambda_h L(\rho)$ is positive semidefinite,
finishing the proof.
\end{proof}

If the potential vanishes, we can define $w_i=1$ for all $i=0,\ldots,n$.
Then the entropy
$$
  \en(\rho) = \sum_{i=0}^n f(\rho_i)\quad\mbox{with }f'(s)=\log\phi(s)
$$
is displacement convex with respect to $\W$. The following remark, based on an
idea of \cite{ErMa14}, shows that
this result may not hold for other entropies.

\begin{remark}\rm\label{rem.alter}
Erbar and Maas \cite{ErMa14} considered the diffusion equation in the form
$$
  \pa_t\rho = \Delta\phi(\rho) = \diver(\rho\na U'(\rho)),
$$
where $U$ satisfies $sU''(s)=\phi'(s)$. The corresponding numerical scheme becomes
$$
  \pa_t\rho = -K(\rho)U'(\rho), \quad K(\rho)=G^\top L(\rho)G,
$$
where $U'(\rho)=(U'(\rho_0),\ldots,U'(\rho_n))$ and the operator $L(\rho)$ is
again defined by $L(\rho)=\diag(\Lambda(\rho_i,\rho_{i+1}))$, but with the mean 
function
\begin{equation}\label{4.EM}
  \Lambda(\rho_i,\rho_{i+1}) 
	= \frac{\phi(\rho_i)-\phi(\rho_{i+1})}{U'(\rho_i)-U'(\rho_{i+1})}.
\end{equation}
The associated entropy is $\en(\rho)=\sum_{i=0}^n U(\rho_i)$,
and if $\rho$ is a geodesic curve on $X_n$ with respect to the nonlinear
transportation metric $\W$ induced by \eqref{4.EM}, then
$$
  \frac{d^2}{dt^2}\en(\rho) = \frac12\langle \widetilde M(\rho)G\psi,G\psi\rangle
$$
where $\widetilde M = DL(\rho)L(\rho)[Q\phi(\rho)] + G\Phi'(\rho)G^\top L(\rho)
+ L(\rho)G\Phi'(\rho)G^\top$. In fact, $\widetilde M$ is the tridiagonal matrix
$$
  \widetilde M = \frac{1}{h^2}\begin{pmatrix}
	d_1 & c_1 & 0   & \cdots & 0 \\
	c_1 & d_2 & c_2 & \ddots & \vdots \\
	0   & c_2 & \ddots &     & 0 \\
	\vdots & \ddots &  & d_{n-1} & c_{n-1} \\
	0   & \cdots & 0 & c_{n-1}   & d_n 
	\end{pmatrix}, 
$$
with the matrix coefficients 
\begin{align*}
  d_i &= 2\Lambda(\rho_i,\rho_{i+1})\big(\phi'(\rho_i)+\phi'(\rho_{i+1})\big)
	+ \pa_1\Lambda(\rho_i,\rho_{i+1})\big(\phi(\rho_{i-1}) - 2\phi(\rho_i)
	+\phi(\rho_{i+1})\big) \\
	&\phantom{xx}{}+ \pa_2\Lambda(\rho_i,\rho_{i+1})\big(\phi(\rho_i)
	-2\phi(\rho_{i+1})+\phi(\rho_{i+2})\big), \quad i=1,\ldots,n, \\
	c_i &= -\phi'(\rho_{i+1})\big(\Lambda(\rho_i,\rho_{i+1})
	+ \Lambda(\rho_{i+1},\rho_{i+2})\big), \quad i=1,\ldots,n-1.
\end{align*}
If $\phi(s)=s^2$, we have $\Lambda(s,t)=(s+t)/2$ and the second principal minor
equals
\begin{align*}
  d_1d_2-c_1^2
	&= \frac12\rho_0^2\rho_1^2 + \frac32\rho_0^2\rho_2^2 + 4\rho_0^2\rho_2\rho_3
	+ \frac32\rho_0^2\rho_3^2 + \frac12\rho_0^2\rho_4^2 + \rho_0\rho_1^3
	+ 3\rho_0\rho_1\rho_2^2 \\
	&\phantom{xx}{}+ 8\rho_0\rho_1\rho_2\rho_3 + 3\rho_0\rho_1\rho_3^2
	+ \rho_0\rho_1\rho_4^2 + \frac14\rho_1^4 + 2\rho_1^2\rho_2\rho_3
	+ \frac34\rho_1^2\rho_3^2 + \frac14\rho_1^2\rho_4^2 \\
	&\phantom{xx}{}- 4\rho_1\rho_2^3
	- 2\rho_1\rho_2^2\rho_3 - \frac{13}{4}\rho_2^4 - 2\rho_2^3\rho_3
	- \frac14\rho_2^2\rho_3^2 + \frac14\rho_2^2\rho_4^2.
\end{align*}
The coefficient 13/4 of the highest power in $\rho_2$ is negative and therefore,
the second principal minor may be negative. According to Sylvester's criterion,
$\widetilde M$ is not positive semidefinite. For instance, choosing special 
initial data, the entropy fails to be convex at time $t=0$. 
\qed
\end{remark}


\begin{appendix}
\section{Properties of mean functions}\label{sec.mean}

We need some properties of the mean function
\begin{equation}\label{a.La}
  \Lambda^f(s,t) = \frac{s-t}{f'(s)-f'(t)} \quad\mbox{for }s\neq t, \quad
	\Lambda^f(s,s) = \frac{1}{f''(s)},
\end{equation}
which we recall here. First, we are concerned
with the logarithmic mean, i.e.\ $f'(s)=\log s$, for which we write 
simply $\Lambda$.

\begin{lemma}[Properties of the logarithmic mean]\label{lem.Lambda}
For all $s$, $t>0$, we have
\begin{align*}
  {\rm (i)}\qquad & 
	\Lambda(s,t)=\Lambda(t,s), \quad \pa_1\Lambda(s,t)=\pa_2\Lambda(t,s), \\
	{\rm (ii)}\qquad & 
	\pa_1\Lambda(s,t) = \frac{\Lambda(s,t)(s-\Lambda(s,t))}{s(s-t)}, \quad
	s\neq t, \\
	{\rm (iii)}\qquad &
	\pa_1\Lambda(s,t)+\pa_2\Lambda(s,t) = \frac{\Lambda(s,t)^2}{st}, \\
	{\rm (iv)}\qquad & 
	\max_{r\ge 0}\big(\Lambda(r,t)-\pa_1\Lambda(t,s)r\big) = t\pa_1\Lambda(s,t), \\
	{\rm (v)}\qquad &
	\Lambda(s,t)\bigg(\frac{a}{s}+\frac{b}{t}\bigg) \ge 2\sqrt{ab}\quad
	\mbox{for }a,b>0.
\end{align*}
\end{lemma}

\begin{proof}
Properties (i)-(iii) can be easily verified by a calculation.
Properties (iv)-(v) are shown in \cite[Appendix A]{Mie13}.
\end{proof}

\begin{lemma}\label{lem.second}
Let $\Lambda\in C^1([0,\infty)^2)$ be any function being
concave in both variables, and let $u_0,u_1,u_2,u_3\ge 0$. Then
\begin{align}
  -\Lambda(u_0,u_1) &+ 2\Lambda(u_1,u_2) - \Lambda(u_2,u_3) \nonumber \\
	&\ge \pa_1\Lambda(u_1,u_2)(-u_0+2u_1-u_2) + \pa_2\Lambda(u_1,u_2)(-u_1+2u_2-u_3).
	\label{soineq}
\end{align}
\end{lemma} 

\begin{proof}
Since $\Lambda$ is concave in both variables, we have
\begin{align*}
  \Lambda(u_0,u_1)-\Lambda(u_1,u_2) &\le \pa_1\Lambda(u_1,u_2)(u_0-u_1)
	+ \pa_2\Lambda(u_1,u_2)(u_1-u_2), \\
	\Lambda(u_2,u_3)-\Lambda(u_1,u_2) &\le \pa_1\Lambda(u_1,u_2)(u_2-u_1)
	+ \pa_2\Lambda(u_1,u_2)(u_3-u_2),
\end{align*}
and adding both inequalities gives the conclusion.
\end{proof}

\begin{lemma}[Concavity of mean functions]\label{lem.mean}
Let $\Lambda^f:[0,\infty)^2\to\R$ be given by \eqref{a.La} and let either
$f(s)=s(\log s-1)$ or $f(s)=s^\alpha$,
where $1<\alpha\le 2$. Then $\Lambda^f$ is concave in both variables.
\end{lemma}

\begin{proof}
For $f(s)=s(\log s-1)$, we refer to \cite[Section 2]{ErMa12}. The statement
for $f(s)=s^\alpha$ is proved in \cite[Appendix]{JuYu15}.
\end{proof}


\section{A priori estimates}\label{sec.apriori}

\begin{lemma}[A priori estimates]\label{lem.est}
Let $ \phi$ be nondecreasing, $h>0$ and let 
$\rho=(\rho_0,\ldots,\rho_n)\in C^1([0,T^*];\R^{n+1})$ for some $T^*>0$ 
be the solution to 
\begin{equation}\label{a.eq}
  h^2\pa_t\rho_i = \phi(\rho_{i-1}) - 2\phi(\rho_i) + \phi(\rho_{i+1}),
	\quad i=0,\ldots,n,
\end{equation}
where $\rho_{-1}=\rho_0$ and $\rho_{n+1}=\rho_n$. Then, for all 
$i=0,\ldots,n$ and $t>0$,
\begin{align}
  & \min_{i=0,\ldots,n}\rho_i(0)
	\le \rho_i(t) \le \max_{i=0,\ldots,n}\rho_i(0), \label{est.max} \\
	& \max_{i=0,\ldots,n}|\na_h\phi(\rho_i(t))| 
	\le h^{-1/2}|\na_h\phi(\rho_i(0))|_2, \label{est.grad}
\end{align}
where $\na_h\phi(\rho_i(t)) = h^{-1}(\phi(\rho_{i+1}(t))-\phi(\rho_i(t)))$ and
\begin{equation}\label{norm}
  |\na_h\phi(\rho_i(0))|_2
	:= \bigg(\sum_{i=0}^n h\big|\na_h\phi(\rho_i(0))\big|^2\bigg)^{1/2}.
\end{equation}
\end{lemma}

\begin{proof}
We multiply \eqref{a.eq} by $(\rho_i-M)_+=\max\{0,\rho_i-M\}$ and
sum over $i=0,\ldots,n$:
\begin{align*}
  \frac{h^2}{2}\pa_t & \sum_{i=0}^n(\rho_i-M)_+^2 \\
	&= \sum_{i=0}^n\big(\phi(\rho_{i-1})-\phi(\rho_i)\big)(\rho_i-M)_+ 
	- \sum_{i=0}^n\big(\phi(\rho_i)-\phi(\rho_{i+1})\big)(\rho_i-M)_+ \\
	&= \sum_{j=0}^n\big(\phi(\rho_{j})-\phi(\rho_{j+1})\big)(\rho_{j+1}-M)_+ 
	- \sum_{i=0}^n\big(\phi(\rho_i)-\phi(\rho_{i+1})\big)(\rho_i-M)_+ \\
	&= -\sum_{i=0}^n\big(\phi(\rho_i)-\phi(\rho_{i+1})\big)\big((\rho_i-M)_+
	- (\rho_{i+1}-M)_+\big) \le 0,
\end{align*}
since $\phi$ is nondecreasing. This shows that
$$
  \sum_{i=0}^n(\rho_i(t)-M)_+^2 \le \sum_{i=0}^n(\rho_i(0)-M)_+^2.
$$
Thus, if $M=\max_{i=0,\ldots,n}\rho_i(0)$, the upper bound in \eqref{est.max}
follows. The lower bound is proved analogously.

For the proof of \eqref{est.grad}, we compute
\begin{align*}
  \frac{h^2}{2}\pa_t & \sum_{i=0}^n \big(\phi(\rho_{i+1})-\phi(\rho_i)\big)^2
	= h^2\sum_{i=0}^n(\phi(\rho_{i+1})-\phi(\rho_i)\big)
	\big(\phi'(\rho_{i+1})\pa_t\rho_{i+1}-\phi'(\rho_i)\pa_t\rho_i\big) \\
	&= \sum_{i=0}^n(\phi(\rho_{i+1})-\phi(\rho_i)\big)
	\phi'(\rho_{i+1})\big(\phi(\rho_{i}) - 2\phi(\rho_{i+1}) + \phi(\rho_{i+2})\big) \\
	&\phantom{xx}{}- \sum_{i=0}^n\big(\phi(\rho_{i+1})-\phi(\rho_i)\big) 
	\phi'(\rho_i)\big(\phi(\rho_{i-1}) - 2\phi(\rho_{i}) + \phi(\rho_{i+1})\big).
\end{align*}
Making the change of variables $i\mapsto i-1$ in the first sum and rearranging
the terms, we find that
$$
  \frac{h^2}{2}\pa_t\sum_{i=0}^n\big(\phi(\rho_{i+1})-\phi(\rho_i)\big)^2
	= -\sum_{i=0}^n\phi'(\rho_i)
	\big(\phi(\rho_{i-1}) - 2\phi(\rho_{i}) + \phi(\rho_{i+1})\big)^2 \le 0.
$$
Consequently, for any $j=0,\ldots,n-1$ and $t>0$,
\begin{align*}
  \big(\phi(\rho_{j+1}(t))-\phi(\rho_j(t))\big)^2
	&\le \sum_{i=0}^n\big(\phi(\rho_{i+1}(t))-\phi(\rho_i(t))\big)^2 \\
	&\le \sum_{i=0}^n\big(\phi(\rho_{i+1}(0))-\phi(\rho_i(0))\big)^2
	= h|\na_h\phi(\rho_i(0))|_2^2.
\end{align*}
Taking the maximum over $j=0,\ldots,n-1$ shows \eqref{est.grad}.
\end{proof}

\begin{corollary}\label{coro1}
Let $\phi$ be nondecreasing and invertible, 
$h>0$, and let $\rho=(\rho_0,\ldots,\rho_n)$ 
be the solution to \eqref{a.eq}. 
We assume that $m:=\min_{i=0,\ldots,n}\rho_i(0)>0$ and set
$M:=\max_{i=0,\ldots,n}\rho_i(0)$. Then
\begin{equation}\label{est.grad2}
  \max_{i=0,\ldots,n}|\na_h\phi'(\rho_i)|
	\le h^{-1/2}\max_{s\in[\phi^{-1}(m),\phi^{-1}(M)]}
	\bigg|\frac{\phi''(s)}{\phi'(s)}\bigg|
	|\na_h\rho_i(0)|_2,
\end{equation}
where $|\na_h\rho_i(0)|_2$ is defined in \eqref{norm}.
\end{corollary}

\begin{proof}
First, note that $m\le\rho_i(t)\le M$ for all $i=0,\ldots,n$ and $t>0$,
by Lemma \ref{lem.est}.
Then the result follows from the mean value theorem. Indeed, we have for some
$\xi$ between $\rho_{i+1}$ and $\rho_i$,
\begin{align*}
  h^{-1}\big|\phi'(\rho_{i+1})-\phi'(\rho_i)\big|
	&= \frac{1}{h}\big|(\phi'\circ\phi^{-1})(\phi(\rho_{i+1}))-(\phi'\circ\phi^{-1})
	(\phi(\rho_i))\big| \\
	&= \bigg|\frac{\phi''(\phi^{-1}(\xi))}{\phi'(\phi^{-1}(\xi))}\bigg|
	\big|\phi(\rho_{i+1})-\phi(\rho_i)\big| \\
	&\le \frac{1}{h}\max_{s\in[\phi^{-1}(m),\phi^{-1}(M)]}
	\bigg|\frac{\phi''(s)}{\phi'(s)}\bigg|
	\max_{j=0,\ldots,n}\big|\phi(\rho_{j+1})-\phi(\rho_j)\big|,
\end{align*}
and we conclude after applying \eqref{est.grad}.
\end{proof}

\begin{example}\label{coro2}\rm
Let $\phi(s)=s^\alpha$ for $\alpha\in(0,1)$, $h>0$ and 
let $\rho=(\rho_0,\ldots,\rho_n)$ be the solution to \eqref{a.eq}
with $m:=\min_{i=0,\ldots,n}\rho_i(0)>0$ and
$M:=\max_{i=0,\ldots,n}\rho_i(0)$. We claim that
$$
  \min_{i=0,\ldots,n}\phi'(\rho_i) 
	\le M^{\alpha-1}, \quad
	\max_{i=0,\ldots,n}|\na_h\phi'(\rho_i)| \le (1-\alpha)m^{-2/\alpha}h^{-1/2}
	|\na_h\rho_i(0)|_2,
$$
where $|\na_h\rho_i(0)|_2$ is defined in \eqref{norm}.
Indeed, the first statement follows from $\alpha<1$ and \eqref{est.max}:
$$
  \min_{i=0,\ldots,n}\phi'(\rho_i) 
	= \Big(\max_{i=0,\ldots,n}\rho_i\Big)^{\alpha-1} \le M^{\alpha-1},
$$
and the second statement is a consequence of Corollary \ref{coro1}
evaluating the right-hand side of \eqref{est.grad2}.
\end{example}

\end{appendix}


\end{document}